\documentclass[12pt,reqno]{amsart}

\usepackage{amssymb}
\usepackage{amsthm}
\usepackage{amsmath}
\usepackage{a4wide}
\usepackage{latexsym}
\usepackage{mathrsfs}
\usepackage[mathx]{mathabx}
\usepackage[v2,tips]{xy}
\usepackage{float}
\usepackage[T1]{fontenc}
\usepackage[latin1]{inputenc}

\newtheorem{theorem}{Theorem}[section]

\theoremstyle{definition}
\newtheorem*{note}{Note}
\newtheorem*{noteadded}{Note added in proof}

\numberwithin{equation}{section}

\newcounter{smallromans}

\newenvironment{romanenumerate}
{\begin{list}{{\normalfont\textrm{(\roman{smallromans})}}}%
    {\usecounter{smallromans}\setlength{\itemindent}{0cm}%
      \setlength{\leftmargin}{5.5ex}\setlength{\labelwidth}{5.5ex}%
      \setlength{\topsep}{0.2ex}\setlength{\partopsep}{0ex}%
      \setlength{\itemsep}{0.2ex}}}%
  {\end{list}}

\newcommand{\romanref}[1]{{\normalfont\textrm{(\ref{#1})}}}

\newcounter{smallalphs}

  {\end{list}}

\renewcommand{\le}{\ensuremath{\leqslant}}

\renewcommand{\ge}{\ensuremath{\geqslant}}

\newcommand{\N}{\mathbb{N}}

\newcommand{\R}{\mathbb{R}}
\newcommand{\C}{\mathbb{C}}
\newcommand{\K}{\mathbb{K}}

\renewcommand{\phi}{\ensuremath{\varphi}}
\renewcommand{\epsilon}{\ensuremath{\varepsilon}}

\newcommand{\clspa}{\overline{\operatorname{span}}\,}
\newcommand{\ad}{\operatorname{Ad}}

\newcommand{\smashw}[2][l]{{\text{\makebox[0pt][#1]{$#2$}}}}

\begin{document}
\title[Extensions and the weak Calkin algebra of Read's Banach
  space]{Extensions and the weak Calkin algebra of Read's Banach space
  admitting discontinuous derivations}
\dedicatory{In memoriam: Charles J.~Read (1958--2015)} 
\subjclass[2010]%
{Primary
46H10, 
  46M18, 
  47L10; 
   Secondary 
  16S70} 
\author[N.~J.~Laustsen]{Niels Jakob Laustsen} \address{Department of
  Mathematics and Statistics, Fylde College, Lancaster University,
  Lancaster LA1 4YF, United Kingdom}
\email{n.laustsen@lancaster.ac.uk}
\author[R.~Skillicorn]{Richard Skillicorn} \address{Department of
  Mathematics and Statistics, Fylde College, Lancaster University,
  Lancaster LA1 4YF, United Kingdom}
\email{r.skillicorn@lancaster.ac.uk}
\begin{abstract} 
Read produced the first example of a Banach space~$E_{\text{R}}$ such
that the associated Banach algebra~$\mathscr{B}(E_{\text{R}})$ of
bounded operators admits a discontinuous derivation (\emph{J.~Lon\-don
  Math.\ Soc.}~1989).  We generalize Read's main theorem
about~$\mathscr{B}(E_{\text{R}})$ from which he deduced this
conclusion, as well as the key technical lemmas that his proof relied
on, by constructing a strongly split-exact sequence
\[ \{0\}\longrightarrow\mathscr{W}(E_{\text{R}})%
\longrightarrow\mathscr{B}(E_{\text{R}})%
   \mathrel{\substack{{\displaystyle{\longrightarrow}}\\[-.5ex]%
       {\displaystyle{\longleftarrow}}}}
   \ell_2^\sim\longrightarrow\{0\}, \] 
where $\mathscr{W}(E_{\text{R}})$ denotes the ideal of weakly compact
operators on~$E_{\text{R}}$, while~$\ell_2^\sim$ is the unitization of 
the Hilbert space~$\ell_2$, endowed with the zero product.\bigskip

\noindent%
To appear in \emph{Studia Mathematica.} 
\end{abstract}
\maketitle
\section{Introduction and statement of the main result}%
\label{section1}
\noindent 
In 1989, Read~\cite{read} published the construction of a remarkable
Banach space that we shall denote by~$E_{\text{R}}$. Read's purpose
was to produce an example of a discontinuous derivation from the
Banach algebra~$\mathscr{B}(E)$ of bounded operators on a Banach
space~$E$, thus answering an open question in automatic continuity
theory going back at least to~\cite{johnson}, in which Johnson had
shown that, for each Banach space~$E$ that is isomorphic to its
Cartesian square~\mbox{$E\oplus E$} (or, more generally, such that~$E$
has a `continued bisection'), every homomorphism from~$\mathscr{B}(E)$
into a Banach algebra is continuous. A general result states that the
continuity of all homomorphisms from a Banach algebra implies the
continuity of all derivations from it; see, \emph{e.g.,}
\cite[Theorem~2.7.5(i)]{dales}. Hence Read's result implies that
Johnson's theorem does not extend to all Banach spaces, and it was the
first example of a Banach space~$E$ for which~$\mathscr{B}(E)$ admits
a discontinuous homomorphism.

Read's strategy was to establish the following theorem, from which the
existence of a discontinuous derivation
from~$\mathscr{B}(E_{\text{R}})$ into a one-dimensional Banach
$\mathscr{B}(E_{\text{R}})$-bimodule follows by standard methods, as
detailed in \cite[Theorem~1]{read}.

\begin{theorem}[Read]\label{readMainThm}
There exists a Banach space~$E_{\normalfont{\text{R}}}$ such that the
Banach algebra $\mathscr{B}(E_{\normalfont{\text{R}}})$ contains a
closed ideal~$\mathscr{I}$ of codimension one, but
the closed linear span of its products,
\[ \mathscr{I}^2 = \clspa\{ ST : S,T\in\mathscr{I}\}, \]
has infinite codimension in~$\mathscr{B}(E_{\normalfont{\text{R}}})$. 
More precisely, 
\begin{romanenumerate}
\item\label{readMainThm1} $\mathscr{I}$ contains the ideal
  $\mathscr{W}(E_{\normalfont{\text{R}}})$ of weakly compact operators
  on~$E_{\normalfont{\text{R}}};$
\item\label{readMainThm2} $\mathscr{W}(E_{\normalfont{\text{R}}})$ has
  infinite codimension in~$\mathscr{B}(E_{\normalfont{\text{R}}});$
\item\label{readMainThm3}
  $\mathscr{I}^2\subseteq\mathscr{W}(E_{\normalfont{\text{R}}}).$
\end{romanenumerate}
\end{theorem}

This theorem will be an immediate consequence of our main result,
which will also incorporate and strengthen the key technical lemmas
that Read used to establish it, as we shall explain below, once we
have stated our main result precisely. It involves the following
notation. Endow the separable, infinite-dimensional Hilbert
space~$\ell_2$ with the zero product, and denote its unitization
by~$\ell_2^{\sim}$, so that $\ell_2^{\sim} = \ell_2\oplus\K 1$ as a
vector space (where~$\K$ denotes the scalar field, either~$\R$
or~$\C$, and~$1$ is the formal identity that we adjoin), and the
product and norm on~$\ell_2^{\sim}$ are given by
\[ (\xi + \lambda
1)(\eta+\mu 1) = \lambda \eta+\mu \xi+ \lambda\mu 1\quad
\text{and}\quad \|\xi + \lambda 1\| = \|\xi\| + |\lambda|\qquad
(\xi,\eta\in\ell_2,\, \lambda,\mu\in\mathbb{K}). \]%
\begin{theorem}\label{WEBEsplitexact} There exists a continuous, 
surjective homomorphism~$\psi$ from the Banach
algebra~$\mathscr{B}(E_{\normalfont{\text{R}}})$ onto~$\ell_2^{\sim}$
with $\ker\psi = \mathscr{W}(E_{\normalfont{\text{R}}})$ such that the
short-exact sequence
\[ \spreaddiagramcolumns{2ex}\xymatrix{\{0\}\ar[r] &
    \mathscr{W}(E_{\normalfont{\text{R}}})\ar[r] &
    \mathscr{B}(E_{\normalfont{\text{R}}})\ar^-{\displaystyle{\psi}}[r]
    & \ell_2^{\sim}\ar[r] & \{0\}} \]
splits strongly, in the sense that there is a
continuous homomorphism from~$\ell_2^{\sim}$ into
$\mathscr{B}(E_{\normalfont{\text{R}}})$ which is a right inverse
of~$\psi$.
\end{theorem}

Taking $\mathscr{I}$ to be the preimage under~$\psi$ of the
codimension-one ideal~$\ell_2$ of~$\ell_2^{\sim}$, we see that
clauses~\romanref{readMainThm1}--\romanref{readMainThm3} of
Theorem~\ref{readMainThm} are satisfied, so
that Theorem~\ref{readMainThm} follows from Theorem~\ref{WEBEsplitexact},
as claimed above.  To explain how Theorem~\ref{WEBEsplitexact} 
incorporates and strengthens the key technical lemmas in Read's
construction, let us first describe the latter in more detail.

After defining the Banach space~$E_{\normalfont{\text{R}}}$, Read
begins his study of it by showing that the
quotient~$E_{\normalfont{\text{R}}}^{**}/E_{\normalfont{\text{R}}}$ is
isomorphic to a Hilbert space~$H$, and he then identifies a particular
ortho\-normal basis $(a_n)_{n\in\N}$ for this Hilbert space (see
\cite[equation~(3.6.5) and Lemma~3.7]{read}).  Given an operator
\mbox{$T\in\mathscr{B}(E_{\normalfont{\text{R}}})$}, its
bidual~$T^{**}$ induces an operator~$\Theta_0(T)$
on~$E_{\normalfont{\text{R}}}^{**}/E_{\normalfont{\text{R}}}$, and
hence on~$H$, by a standard construction, as detailed in the
diagram~\eqref{defnTheta0} below. Read's main technical achievement is
\cite[Lemma~4.1]{read}, which states that when~$\Theta_0(T)$ is
considered as an operator on~$H$, its matrix~$M$ with respect to the
orthonormal basis $(a_n)_{n\in\N}$ has a very special form, namely
\begin{equation}\label{ReadEq411}
M = \begin{pmatrix} \lambda & 0 & 0 & \cdots\\
\mu_1 & \lambda & 0 & 0 & \cdots\\
\mu_1 & 0 & \lambda & 0 & 0 & \cdots\\
\mu_2 & 0 & 0 & \lambda & 0 & 0 & \cdots\\
\mu_2 & 0 & 0 & 0 & \lambda & 0 & 0 & \cdots\\
\mu_3 & 0 & 0 & 0 & 0 & \lambda & 0 & 0 & \cdots\\
\mu_3 & 0 & 0 & 0 & 0 & 0 & \lambda & 0 & 0 & \cdots\\
\vdots & \vdots & \vdots & \vdots & \vdots & \vdots & & \ddots
\end{pmatrix}
\end{equation}
for some scalars $\lambda$ and $\mu_1,\mu_2,\mu_3,\ldots$\, (Note
that~\eqref{ReadEq411} corrects a typo in
\cite[equation~(4.1.1)]{read}: the first entry of the fifth row of~$M$
should be~$\mu_2$, not~$\mu_3$, as \cite[Lemma~4.1(c)]{read} shows.)
Finally, in \cite[Lemma~4.2]{read}, Read establishes a partial
converse to this result by showing that, in the case where only
finitely many of the scalars $\mu_1,\mu_2,\mu_3,\ldots$ are non-zero,
\eqref{ReadEq411} arises as the matrix of~$\Theta_0(T)$ for some
operator $T\in\mathscr{B}(E_{\normalfont{\text{R}}})$.

The connection between these results and Theorem~\ref{WEBEsplitexact}
goes via the following ob\-ser\-va\-tion. The first column of the
matrix~\eqref{ReadEq411} is the image of the first basis vector under
the operator~$\Theta_0(T)$, so that the sequence~$(\mu_i)_{i\in\N}$
belongs to~$\ell_2$. Hence we can define a mapping
$\psi\colon\mathscr{B}(E_{\normalfont{\text{R}}})\to\ell_2^{\sim}$ by
$\psi(T) = (\mu_i)_{i\in\N}+\lambda 1$, where $\lambda$ and
$\mu_1,\mu_2,\ldots$ are the scalars determined by~$\Theta_0(T)$
via~\eqref{ReadEq411}. We shall show that this mapping~$\psi$ is the
surjective homomorphism whose existence is stated in
Theorem~\ref{WEBEsplitexact}, which therefore strengthens Read's
technical lemmas in two ways. First, the surjectivity of~$\psi$ means
that all possible matrices of the form~\eqref{ReadEq411} arise as the
matrix of~$\Theta_0(T)$ for some operator
$T\in\mathscr{B}(E_{\normalfont{\text{R}}})$, not just those whose
first column vanishes eventually. Second, we can choose a
pre\-image~$T$ under~$\psi$ of the element $(\mu_i)_{i\in\N}+\lambda
1\in\ell_2^{\sim}$ in such a way that the corresponding mapping is a
bounded homomorphism.

Our motivation for proving Theorem~\ref{WEBEsplitexact} is that it
has enabled us to show that:
\begin{itemize}
\item the Banach algebra~$\mathscr{B}(E_{\text{R}})$ has a singular
  extension which splits algebraically, but it is not admissible, and
  so does not split strong\-ly;
\item the homological bidimension of~$\mathscr{B}(E_{\text{R}})$ is at
  least two.
\end{itemize}
The first of these results solves a natural problem left open in Bade,
Dales, and Lykova's comprehensive study~\cite{bdl} of splittings of
extensions of Banach algebras, while the second answers a question
originating in Helemskii's seminar at Moscow State University. For
details of any unexplained terminology and how to deduce these results
from Theorem~\ref{WEBEsplitexact}, we refer to~\cite{NJLRS}.

\begin{noteadded} After the completion of~\cite{NJLRS}, in
joint work with Kania~\cite{KLS}, we have succeeded in constructing a
singular extension of~$\mathscr{B}(E_{\text{R}})$ which is admissible
and splits algebraically, but does not split strong\-ly, thus
complementing the first of the two results mentioned above. Once
again, Theorem~\ref{WEBEsplitexact} plays a key role in the proof of
this result.
\end{noteadded}

To conclude this Introduction, we would like to express the hope that,
by elucidating the structure of the Banach
algebra~$\mathscr{B}(E_{\text{R}})$, Theorem~\ref{WEBEsplitexact} will
spark new interest in and lead to further applications of this
remarkable creation of Read's.

\section{Read's Banach space~$E_{\normalfont{\text{R}}}$ and the proof of
  Theorem~{\normalfont{\ref{WEBEsplitexact}}}}\label{section3}
\noindent
We begin this section with some general conventions and results,
followed by an overview of Read's construction of the Banach
space~$E_{\normalfont{\text{R}}}$ and details of the specific elements
of Read's work that we shall require in our proof of
Theorem~\ref{WEBEsplitexact}, before we conclude with the proof
itself. Our aim is to provide enough detail to make this presentation
self-contained, without repeating arguments already given
in~\cite{read}.  An expanded, entirely self-contained version of the
proof of Theorem~\ref{WEBEsplitexact}, incorporating all necessary
details of Read's construction, can be found
in~\cite[Chapter~5]{skillicorn}.

We shall generally follow the notation and terminology used
in~\cite{read}; in some places, however, we add extra details or take
a slightly different view from Read's in order to facilitate our proof
of Theorem~\ref{WEBEsplitexact} and avoid ambiguities.

All results in~\cite{read} are stated for complex scalars only. We
observe that the proofs carry over verbatim to the real case, so we
shall address both cases simultaneously, denoting the scalar field
by~$\K$, so that either $\mathbb{K} = \R$ or $\mathbb{K} = \C$.

By an \emph{operator,} we understand a bounded, linear mapping between
Banach spaces. For a Banach space~$E$, we denote by~$E^*$ its dual
space, and we identify~$E$ with its canonical image in the bidual
space~$E^{**}$. Let $\pi_E\colon E^{**}\to E^{**}/E$ denote the
quotient mapping. Since the restriction to~$E$ of the bidual~$T^{**}$
of an operator~$T\in\mathscr{B}(E)$ is equal to~$T$, it leaves the
subspace~$E$ invariant, and hence the Fundamental Isomorphism Theorem
implies that there is a unique operator
$\Theta_0(T)\in\mathscr{B}(E^{**}/E)$ such that the diagram
\begin{equation}\label{defnTheta0}
\spreaddiagramcolumns{6ex}\spreaddiagramrows{2ex}\xymatrix{%
E^{**}\ar^-{\displaystyle{T^{**}}}[r]\ar_-{\displaystyle{\pi_E}}[d] &
E^{**}\ar^-{\displaystyle{\pi_E}}[d]\\
E^{**}/E\ar@{-->}^-{\displaystyle{\Theta_0(T)}}[r] & E^{**}/E}
\end{equation}
is commutative. Moreover, the mapping $\Theta_0\colon
T\mapsto\Theta_0(T),\,\mathscr{B}(E)\to\mathscr{B}(E^{**}/E)$, is a
continuous, unital algebra homomorphism of norm at most one, and a standard
characterization of the ideal of weakly compact operators shows that
$\ker\Theta_0 = \mathscr{W}(E)$.

Let~$(A,\|\cdot\|_A)$ be a Banach space with a normalized, symmetric
basis~$(e_n)_{n\in\N}$. By passing to an equivalent
norm, we may suppose that the basis $(e_n)_{n\in\N}$ is
$1$-symmetric, that is, for each $n\in\N$, each
permutation~$\pi$ of~$\N$, and all scalars
$\lambda_1,\ldots,\lambda_n,\mu_1,\ldots,\mu_n$, we have
\[
\biggl\| \sum_{j=1}^n\lambda_j\mu_j e_{\pi(j)}\biggr\|_A\le\max\bigl\{
|\mu_1|,\ldots,|\mu_n|\bigr\}\biggl\| \sum_{j=1}^n\lambda_j
e_j\biggr\|_A. \] Read \cite[Definition~1.2(b)]{read} defines
the \emph{James-like space}~$JA$ based on~$A$ as the collection of all
scalar sequences $x = (\lambda_n)_{n\in\N}\in c_0$ such that the
quantity
\begin{multline*} \|x\|_{JA} = 
\sup\Biggl\{\biggl\|\sum_{j=1}^n(\lambda_{p_j}-\lambda_{p_{j+1}})^2e_j
+ \lambda_{p_{n+1}}^2e_{n+1}\biggr\|_A^{\frac{1}{2}}
:\\[-2ex] n,p_1,\ldots,p_{n+1}\in\N,\, p_1<p_2<\cdots<p_{n+1}\Biggr\}
\end{multline*}
is finite, and observes that $(JA,\|\cdot\|_{JA})$ is a Banach space. 

In~\cite[Section~2]{read}, Read proceeds to show that, in the case
where~$A$ contains no subspace isomorphic to~$c_0$, the space~$JA$ is
quasi-reflexive of order one, so that the canonical image of~$JA$ has
codimension one in its bidual~$(JA)^{**}$, just like the original
James space (which in this approach corresponds to $A =
\ell_1$). Moreover, the standard unit vector basis~$(e_n)_{n\in\N}$ is
a shrinking basis for~$JA$, so that the sequence $(e_n^*)_{n\in\N}$ of
coordinate functionals is a basis for the dual space~$(JA)^*$.  Like
Read, we use $(e_n)_{n\in\N}$ to denote the basis of both~$A$
and~$JA$, and we identify~$(JA)^{**}$ with~$JA\oplus\mathbb{K}\Phi$,
where~$\Phi$ is the functional on~$(JA)^*$ given by $\langle
e_n^*,\Phi\rangle = 1$ for each $n\in\N$. We record for later use that
$\|\Phi\|_{(JA)^{**}} =1$ by \cite[Lemma~2.2(a)]{read}.

Let $\N_0 = \N\cup\{0\}$ be the set of non-negative integers, and set
\[ \mathbb{I} = \{2\}\cup\{i\in\N_0 : i\equiv 0\bmod 6,\, i\equiv 
4\bmod 6,\, \text{or}\ i\equiv 5\bmod 6\}. \] For each
$i\in\mathbb{I}$, choose a Banach space~$(B_i,\|\cdot\|_{B_i})$ with a
normalized, $1$-symmetric basis, which we denote by $(e_n)_{n\in\N}$
(independent of~$i$), such that no subspace of~$B_i$ is isomorphic
to~$c_0$, and such that the family $(B_i)_{i\in\mathbb{I}}$ is
\emph{incomparable}, in the sense that, for each $i\in\mathbb{I}$ and
each $\epsilon>0$, there exists $n\in\N$ with
\begin{equation}\label{ReadDefIncomp}
\|\sigma_n\|_{B_i}\le\epsilon\cdot\inf\bigl\{\|\sigma_n\|_{B_j} :
j\in\mathbb{I}\setminus\{i\}\bigr\}, \end{equation} where $\sigma_n = \sum_{k=1}^n
e_k$; Read explains in the note following \cite[Definition~3.2]{read}
that such a family exists. For $i\in\N\setminus\mathbb{I}$, set
\begin{equation}\label{subspaceYiforinotinI} 
 B_i = \begin{cases} B_0 & \text{if}\ i\equiv 1\bmod 6\\ B_2 &
   \text{if}\ i\equiv 2\bmod 6\ \text{or}\ i\equiv 3\bmod
   6, \end{cases} \end{equation} so that $B_i$ is now defined for each
$i\in\N_0$, and we have a sequence $(JB_i)_{i\in\N_0}$ of
James-like Banach spaces, whose $\ell_2$-direct sum $Y =
\bigl(\bigoplus_{i=0}^\infty JB_i\bigr)_{\ell_2}$ is at the heart of
Read's construction.  By standard duality and the quasi-reflexivity of
the spaces~$JB_i$, we have the following identifications:
\begin{equation}\label{dualandbidualofY}
Y^* = \biggl(\bigoplus_{i=0}^\infty (JB_i)^*\biggr)_{\ell_2}\qquad
\text{and}\qquad Y^{**} = \biggl(\bigoplus_{i=0}^\infty
(JB_i)^{**}\biggr)_{\ell_2} = \biggl(\bigoplus_{i=0}^\infty
JB_i\oplus\mathbb{K}\Phi\biggr)_{\ell_2}.
\end{equation}

Another important piece of notation involves the Hilbert
space~$\overline{B} = \ell_2(\N_0)$ and its standard orthonormal
basis~$(b_n)_{n\in\N_0}$, which Read relabels as follows for
positive indices:
\begin{equation}\label{alphabetagammabasis} 
\begin{alignedat}{3} \alpha_n &= b_{6n},\qquad & \beta_n &=
  b_{6(n-1)+1},\qquad & \gamma_n &= b_{6(n-1)+2},\\ \delta_n &=
  b_{6(n-1)+3},\qquad & x_n &= b_{6(n-1)+4},\qquad & y_n &=
  b_{6(n-1)+5} \end{alignedat} \qquad\quad (n\in\N). \end{equation}
For $n\in\N$ and $\xi = \sum_{i=0}^\infty\xi_i b_i\in\overline{B}$,
Read introduces the tensor notation $e_n\otimes\xi =
(\xi_ie_n)_{i=0}^\infty$, which defines an element of~$Y$ due to the
uniform bound $\|e_n\|_{JB_i}\le\|e_1+e_2\|_{B_i}\le 2$, independent
of $i\in\N_0$. This definition extends by linearity to tensors of the
form $x\otimes\xi$ for $x\in c_{00}$.  By assigning special symbols
to the following linear combinations of the basis
vectors~\eqref{alphabetagammabasis},
\begin{equation}\label{dashcombs}
\begin{alignedat}{3}
\alpha_n' &= \alpha_n - (x_n - y_n),\qquad & \beta_n' &= \beta_n -
(x_n + y_n),\\ \gamma_n' &= \gamma_n - (x_n + y_n),\qquad & \delta_n'
&= \delta_n - \Bigl(\frac{b_0}{2^n} - x_n+
y_n\Bigr) \end{alignedat}\qquad\quad (n\in\N), \end{equation} we can
now define Read's space~$E_{\text{R}}$ as follows:
\begin{equation}\label{defnSVN}
\begin{alignedat}{2}
S &= \{\alpha_n', \beta_n', \gamma_n', \delta_n' :
n\in\N\},\qquad\qquad & V &= \clspa S\subseteq\overline{B},\\ N &=
\clspa\{e_n\otimes s: n\in\N,\,s\in S\}\subseteq Y,\qquad\qquad &
E_{\text{R}} &= Y/N.
\end{alignedat}
\end{equation}

\noindent%
\begin{note} Our definition of~$\beta_n'$ above corrects a typo in
\cite[Definition~3.4(b)]{read}, where the sign of~$y_n$ is wrong, as
one can see by comparing it with the second line of the displayed
equations at the bottom of \cite[p.~313]{read} and the seventh
displayed equation of \cite[p.~319]{read}.
\end{note}

As we stated in the Introduction, the first two steps in Read's analysis
of the space~$E_{\text{R}}$ consist of showing that the
quotient~$E_{\text{R}}^{**}/E_{\text{R}}$ is isomorphic to a Hilbert
space and identifying an orthonormal basis for it. Read, however, does
not obtain an explicit formula for this isomorphism. Such a formula
will be required in our proof of Theorem~\ref{WEBEsplitexact}, so we
shall now recast Read's arguments in a form that will produce an
explicit isomorphism~$U$ between the closed subspace
\begin{equation}\label{defnH} H = \clspa\{b_0,x_n,y_n :
n\in\N\} \end{equation} of~$\overline{B}$ 
and~$E_{\text{R}}^{**}/E_{\text{R}}$.

To this end, we observe that the proof of \cite[Lemma~3.7]{read} shows
that the restriction to~$H$ of the quotient mapping
$Q_V\colon\overline{B}\to\overline{B}/V$ is surjective and bounded
below by~$\frac{1}{15}$; that is, $Q_V|_H$ is an isomorphism whose
inverse has norm at most~$15$. (Read denotes this mapping by~$\alpha$;
we prefer~$Q_V|_H$ as it is more descriptive and avoids any possible
confusion with the basis vectors $\alpha_n$.)

Set $U_0 = \pi_{E_{\text{R}}}Q_N^{**}R_0\colon \overline{B}\to
E^{**}_{\text{R}}/E_{\text{R}}$, where \mbox{$\pi_{E_{\text{R}}}\colon
  E^{**}_{\text{R}}\to E^{**}_{\text{R}}/E_{\text{R}}$} and $Q_N\colon
Y\to E_{\text{R}}$ are the quotient mappings, and
$R_0\colon\overline{B}\to Y^{**}$ is the linear isometry given by
\[ R_0\xi = \Phi\otimes\xi = (\xi_i\Phi)_{i=0}^\infty\qquad \biggl(\xi =
\sum_{i=0}^\infty\xi_i b_i\in\overline{B}\biggr), \] using the natural
extension to~$Y^{**}$ of Read's tensor notation for~$Y$ introduced
above. We claim that~$U_0$ is surjective with $\ker U_0 = V$.  To
establish this claim, we first note that $\ker Q_N^{**} =
N^{\circ\circ}$ (the bipolar of~$N$) by a standard duality result, and
therefore \begin{equation}\label{kerpiQN} \ker
  (\pi_{E_{\text{R}}}Q_N^{**}) = Y+N^{\circ\circ}.
\end{equation}
Equation~\eqref{dualandbidualofY} shows that $Y^{**} =
Y+R_0[\overline{B}]$. Combining this with the surjectivity
of~$\pi_{E_{\text{R}}}Q_N^{**}$ and~\eqref{kerpiQN}, we obtain
$E_{\text{R}}^{**}/E_{\text{R}} = \pi_{E_{\text{R}}}Q_N^{**}[Y^{**}] =
U_0[\overline{B}]$, so that~$U_0$ is indeed surjective.

Next, to determine the kernel of~$U_0$, we require Read's observation
\cite[equation~(3.6.3)]{read} that each element~$x^{**}$ of~$Y^{**}$
can be expressed uniquely as $x^{**} = \Phi\otimes\eta_0 +
\sum_{i=1}^\infty e_i\otimes\eta_i$, where $\eta_i\in\overline{B}$
for each $i\in\N_0$.  Using this notation, Read
\cite[Lemma~3.6.4]{read} shows that
\begin{equation}\label{bipolarN} 
  N^{\circ\circ} = \{ x^{**}:
\eta_i\in V\ \text{for each}\ i\in\N_0\}, \end{equation}
which together with~\eqref{kerpiQN}
immediately implies that $V\subseteq\ker U_0$.  

Conversely, suppose that $\xi\in\ker U_0$. Then $\Phi\otimes\xi\in\ker
(\pi_{E_{\text{R}}}Q_N^{**})$, so by~\eqref{kerpiQN}, we can find $y =
\sum_{i=1}^\infty e_i\otimes\eta_i\in Y$ such that $\Phi\otimes\xi -
y\in N^{\circ\circ}$, and hence $\xi\in V$ by~\eqref{bipolarN} (as
well as $\eta_i\in V$ for each $i\in\N$, but we do not need this
information). This proves our claim.

By the Fundamental Isomorphism Theorem, there is a unique
isomorphism~$\widehat{U}_0$ of~$\overline{B}/V$
onto~$E_{\text{R}}^{**}/E_{\text{R}}$ such that $U_0 =
\widehat{U}_0Q_V$, and hence we have a commutative diagram
\begin{equation}\label{commdiagramRead} 
\spreaddiagramcolumns{7ex}\spreaddiagramrows{3.5ex}\xymatrix{%
  H\ar^-{\displaystyle{Q_V|_H}}_-{{\cong}}[r]%
  \ar@{^{(}->}^-{\displaystyle{\iota}}[d] & \overline{B}/V%
  \ar@{-->}_-{{\cong}}^-{\displaystyle{\widehat{U}_0}}[r] &
      E_{\text{R}}^{**}/E_{\text{R}}\\ 
  \overline{B}\ar@{->>}^-{\displaystyle{Q_V}}[ru]%
      \ar^-{\displaystyle{U_0}}[rru]\ar^(0.7){\displaystyle{R_0}}[r]
      & Y^{**}\ar@{->>}^-{\displaystyle{Q_N^{**}}}[r]%
      & E_{\text{R}}^{**}\smashw{,}%
  \ar@{->>}_(0.4){\displaystyle{\pi_{E_{\text{R}}}}}[u]}
\end{equation}
where $\iota\colon H\to\overline{B}$ denotes the
natural inclusion. 

This diagram shows in particular that the restriction~$U$ of the
operator~$U_0$ to the subspace~$H$ is an isomorphism
onto~$E_{\text{R}}^{**}/E_{\text{R}}$, and so it induces a continuous
algebra isomorphism $\ad U\colon T\mapsto U^{-1}TU$ of the Banach
algebra~$\mathscr{B}(E_{\text{R}}^{**}/E_{\text{R}})$
onto~$\mathscr{B}(H)$. Using this notation, we can rephrase
\cite[Lemma~4.1]{read} as follows: for each
$T\in\mathscr{B}(E_{\text{R}})$, the matrix~$M$ of the operator $(\ad
U)\circ\Theta_0(T)$ with respect to the orthonormal basis
$(a_n)_{n\in\N} = (b_0,x_n,y_n)_{n\in\N}$ for~$H$ is given
by~\eqref{ReadEq411}. Alternatively, we can express this identity as
\begin{equation}\label{rangeofadUTheta0new}
(\ad U)\circ\Theta_0(T) = \lambda I_H + \tau_\xi,
\end{equation}
where~$I_H$ is the identity operator on~$H$, $\xi =
\sum_{n=1}^\infty\mu_n(x_n+y_n)$, and~$\tau_\xi$ is the rank-one
operator on~$H$ given by $\eta\mapsto(\eta\,|\,b_0)\xi$, where
$(\eta\,|\,b_0)$ denotes the inner product of~$\eta$ and~$b_0$. We
note in passing that the fact that the spaces $(B_i)_{i\in\mathbb{I}}$
are chosen to be incomparable in the sense of~\eqref{ReadDefIncomp}
plays a crucial role in the proof of \cite[Lemma~4.1]{read}. 

We are now ready to prove Theorem~\ref{WEBEsplitexact}. Our proof
refines that of \cite[Lemma~4.2]{read}, as it is given in
\cite[p.~320]{read}. For clarity, we shall present a fully
self-contained argument.

\begin{proof}[Proof of Theorem~{\normalfont{\ref{WEBEsplitexact}}}] 
For notational convenience, we shall replace the generic Hilbert
space~$\ell_2$ in the statement of Theorem~\ref{WEBEsplitexact}
with the closed subspace
\[ H_0 = \clspa\{x_n+y_n : n\in\N\} \] 
of the Hilbert space~$H$ given by~\eqref{defnH}.  The space~$H_0$ is
of course isometrically isomorphic to~$\ell_2$ via the mapping
$\xi\mapsto \bigl(\frac{1}{\sqrt{2}}(\xi\,|\,x_n+y_n)\bigr)_{n\in\N}$,
so this is really only a change in notation, provided that we
endow~$H_0$ with the zero product. Then, using that the vector~$b_0$ is
orthogonal to~$H_0$, we see that the mapping $\Upsilon\colon \xi
+\lambda1\mapsto\tau_\xi +\lambda I_H$ is a continuous, unital algebra
isomorphism of the unitization $\widetilde{H}_0 = H_0\oplus\K1$
of~$H_0$ (defined analogously to that of~$\ell_2$ given on
page~\pageref{WEBEsplitexact}) onto the closed subalgebra $\mathscr{T} =
\{\tau_\xi +\lambda I_H : \xi\in H_0,\, \lambda\in\K\}$
of~$\mathscr{B}(H)$.

By~\eqref{rangeofadUTheta0new}, the range of $(\ad U)\circ\Theta_0$ is
contained in~$\mathscr{T}$, so we may consider the composite
continuous algebra homomorphism~$\psi$ given by \[ \psi =
\Upsilon^{-1}\circ(\ad U)\circ\Theta_0\colon
\mathscr{B}(E_{\text{R}})\to\widetilde{H}_0. \] Since $\Upsilon^{-1}$
and $\ad U$ are both isomorphisms, we see that $\ker\psi =
\ker\Theta_0 = \mathscr{W}(E_{\text{R}})$. It remains to construct a
continuous algebra homomorphism which is a right inverse of~$\psi$.

Let $\xi = \sum_{n=1}^\infty\xi_n(x_n+y_n)\in H_0$ and $y =
(y(i))_{i=0}^\infty\in Y$ be given, where \mbox{$y(i)\in JB_i$} for
each $i\in\N_0$. By~\eqref{subspaceYiforinotinI}, we
have $JB_{6(n-1)+1} = JB_0$ for each $n\in\N$, so that in analogy with
the tensor notation already introduced, we may
define~$y(0)\otimes\beta_n$ to be the element of~$Y$ whose
$(6(n-1)+1)^{\text{st}}$ coordinate is~$y(0)$, while all other
co\-or\-di\-nates vanish. Then the series
$\sum_{n=1}^\infty\xi_n\,y(0)\otimes\beta_n$ converges in~$Y$, and its
sum has norm
$\frac{1}{\sqrt{2}}\|\xi\|_{\overline{B}}\,\|y(0)\|_{JB_0}$.

Using~\eqref{subspaceYiforinotinI} once more, we see that
$JB_{6(n-1)+2} = JB_{6(n-1)+3} = JB_2$ for each \mbox{$n\in\N$}, so
that we have an absolutely convergent series \[ \sum_{n=1}^\infty
\frac{y(6(n-1)+3)}{2^n} \] in~$JB_2$, whose sum~$y'$ has norm at
most~$\|y\|_Y$. As above, let $y'\otimes\gamma_n$ be the element
of~$Y$ whose \mbox{$(6(n-1)+2)^{\text{nd}}$} coordinate is~$y'$, while
all other co\-or\-di\-nates vanish. Then the series
$\sum_{n=1}^\infty\xi_n\, y'\otimes\gamma_n$ converges in~$Y$, and its
sum has norm at most
$\frac{1}{\sqrt{2}}\|\xi\|_{\overline{B}}\,\|y\|_Y$.

Combining these conclusions, we may define an element of~$Y$ by
\begin{equation}\label{defnTxi}
 T_\xi y = \sum_{n=1}^\infty\xi_n\biggl(y(0)\otimes\beta_n +
 \biggl(\sum_{m=1}^\infty\frac{y(6(m-1)+3)}{2^m}\biggr)\otimes\gamma_n\biggr),
\end{equation}
and $\|T_\xi y\|_Y\le\sqrt{2}\,\|\xi\|_{\overline{B}}\,\|y\|_Y$.  The
mapping $T_\xi\colon y\mapsto T_\xi y$ is clearly linear and has norm
at most~$\sqrt{2}\,\|\xi\|_{\overline{B}}$, and hence we have a
mapping \begin{equation}\label{defnrho0} \rho_0\colon\ \xi\mapsto
  T_\xi,\quad H_0\to\mathscr{B}(Y), \end{equation} which is linear and
bounded with norm at most~$\sqrt{2}$.  

To prove that~$\rho_0$ is multiplicative, we must show that $T_\eta
T_\xi = 0$ for each pair $\xi,\eta\in H_0$ because~$H_0$ has the zero
product. Write~$\xi$ as \mbox{$\xi =
  \sum_{n=1}^\infty\xi_n(x_n+y_n)$}, and let \mbox{$y =
  (y(i))_{i=0}^\infty\in Y$} be given. By~\eqref{defnTxi}, the element
$z = T_\xi y$ has the form $z = (z(i))_{i=0}^\infty$, where
\[ z(i) = \begin{cases} \xi_n y(0)\ &\text{if}\ i =
  6(n-1)+1\ \text{for
    some}\ n\in\N,\\ 
\displaystyle{\xi_n\sum_{m=1}^\infty\frac{y(6(m-1)+3)}{2^m}}\ &\text{if}\ i
  = 6(n-1)+2\ \text{for
    some}\ n\in\N,\\ 0\ &\text{otherwise.} \end{cases} \] Hence $z(0) = 0 = z(6(m-1)+3)$ for each
$m\in\N$, so that another application of~\eqref{defnTxi} shows that $0
= T_\eta z = T_\eta T_\xi y$, as required.

Next, we shall prove that  
\begin{equation}\label{TxiNinvariant}
T_\xi[N]\subseteq N\qquad (\xi\in H_0),
\end{equation}
where~$N$ is the subspace of~$Y$ given by~\eqref{defnSVN}.
Since~$T_\xi$ is bounded and linear, it suffices to show that
$T_\xi(e_n\otimes s)\in N$ for each $n\in\N$ and $s\in S$.
Comparing~\eqref{defnTxi} with~\eqref{alphabetagammabasis}, we see
that $T_\xi(e_n\otimes\eta) = 0$ for each $\eta\in\{\alpha_m, \beta_m,
\gamma_m, x_m, y_m:m\in\N\}$, and therefore $T_\xi(e_n\otimes s) =
0\in N$
for each $s\in\{\alpha_m',\beta_m',\gamma_m':m\in\N\}$
by~\eqref{dashcombs}. Moreover, for $m\in\N$, we have
\begin{align*} T_\xi(e_n\otimes\delta_m') &= T_\xi(e_n\otimes\delta_m) -
\frac{1}{2^m}T_\xi(e_n\otimes b_0) =
\sum_{k=1}^\infty\xi_k\,\frac{e_n}{2^m}\otimes\gamma_k -
\frac{1}{2^m}\sum_{k=1}^\infty\xi_k\, e_n\otimes\beta_k\\ &=
\frac{1}{2^m}\sum_{k=1}^\infty\xi_k\,e_n \otimes(\gamma_k - \beta_k) =
\frac{1}{2^m}\sum_{k=1}^\infty\xi_k\,e_n\otimes(\gamma_k' -
\beta_k')\in N, \end{align*} which completes the proof
of~\eqref{TxiNinvariant}.

Thus, by the Fundamental Isomorphism Theorem, there is a unique
operator \mbox{$\widecheck{T}_\xi\in\mathscr{B}(E_{\text{R}})$} such that
the diagram
\begin{equation}\label{defnTcheck}
\spreaddiagramcolumns{5ex}\spreaddiagramrows{1.5ex}\xymatrix{%
  Y\ar^-{\displaystyle{T_\xi}}[r]\ar_-{\displaystyle{Q_N}}[d] &
  Y\ar^-{\displaystyle{Q_N}}[d]\\ 
  E_{\text{R}}\ar@{-->}^-{\displaystyle{\widecheck{T}_\xi}}[r]
  & E_{\text{R}}}
\end{equation}
is commutative, and $\|\widecheck{T}_\xi\| =
\|Q_NT_\xi\|\le\sqrt{2}\,\|\xi\|_{\overline{B}}$. The fact that the
mapping~$\rho_0$ given by~\eqref{defnrho0} is an algebra
homomorphism implies that the same is true for the mapping
$\xi\mapsto\widecheck{T}_\xi,\, H_0\to\mathscr{B}(E_{\text{R}})$, and
hence
\[ \rho\colon\ \xi+\lambda 1\mapsto \widecheck{T}_\xi + \lambda
I_{E_{\text{R}}},\quad\widetilde{H}_0\to\mathscr{B}(E_{\text{R}}), \]
is a continuous, unital algebra homomorphism, where~$I_{E_{\text{R}}}$
denotes the identity operator on~$E_{\text{R}}$.

We shall now complete the proof by showing that this
homomorphism~$\rho$ is a right
inverse of~$\psi$. Since~$\psi$ and~$\rho$ are both unital, it suffices
to show that $\psi\circ\rho(\xi) = \xi$ for each $\xi\in H_0$, which
amounts to showing that $(\ad U)\circ\Theta_0(\widecheck{T}_\xi) =
\tau_\xi$ by the definitions of~$\rho$, $\psi$, and~$\Upsilon$.
According to~\eqref{rangeofadUTheta0new}, both sides of this identity
belong to~$\mathscr{T}$, so the fact (to be established below) that
\begin{equation}\label{evalatbsuffices}
T = T'\quad \Longleftrightarrow\quad Tb_0 = T'b_0\qquad
(T,T'\in\mathscr{T})
\end{equation}
means that it is enough to verify that $((\ad
U)\circ\Theta_0(\widecheck{T}_\xi))b_0 = \tau_\xi(b_0)$, or, equivalently,
that
\begin{equation}\label{Uxi}
\Theta_0(\widecheck{T}_\xi)Ub_0 = U\xi.
\end{equation}

To prove~\eqref{evalatbsuffices}, let $T = \tau_\eta + \lambda I_H
\in\mathscr{T}$ be given, where $\eta\in H_0$ and
$\lambda\in\mathbb{K}$. Then we have $Tb_0 = \eta + \lambda b_0$,
which uniquely determines both~$\eta$ and~$\lambda$, and hence~$T$,
because~$b_0$ is orthogonal to~$H_0\ni\eta$. This establishes the
implication~$\Leftarrow$, while the converse is clear.

We begin our proof of~\eqref{Uxi} by observing that $\sigma_m =
\sum_{j=1}^m e_j$ is a unit vector in~$JB_i$ for each $m\in\N$ and
$i\in\N_0$, and hence $(\sigma_m\otimes\eta)_{m\in\N}$ is a
norm-bounded sequence in~$Y\subseteq Y^{**}$ for each $\eta =
\sum_{i=0}^\infty\eta_i b_i\in\overline{B}$. Since the elements of the
form $e_n^*\otimes\zeta = (\zeta_ie_n^*)_{i\in\N_0}$, where $n\in\N$
and $\zeta = \sum_{i=0}^\infty\zeta_i b_i\in \overline{B}$, span a
norm-dense subspace of~$Y^*$, and
\[ \langle \sigma_m\otimes\eta, e_n^*\otimes\zeta\rangle = 
\sum_{i=0}^\infty\eta_i\zeta_i = \langle e_n^*\otimes\zeta,
\Phi\otimes\eta \rangle\qquad (m\ge n), \] we conclude that the
sequence $(\sigma_m\otimes\eta)_{m\in\N}$ weak$^*$-converges
to~$\Phi\otimes\eta$ in~$Y^{**}$ by a standard elementary result (see,
\emph{e.g.}, \cite[Exercise~3.3]{fabianetal}), as noted by Read
\cite[p.~315]{read}. 

Writing $\xi$ as $\xi = \sum_{n=1}^\infty\xi_n(x_n+y_n)$, we obtain
\begin{alignat*}{3} Q_N T_\xi(\sigma_m\otimes b_0) &= Q_N\biggl(\sum_{n=1}^\infty
  \xi_n\,\sigma_m\otimes\beta_n\biggr)
  &\quad&\text{by~\eqref{defnTxi}}\\ &= \sum_{n=1}^\infty \xi_n
  Q_N(\sigma_m\otimes\beta_n) = \sum_{n=1}^\infty \xi_n
  Q_N(\sigma_m\otimes(x_n+y_n))
  &\quad&\text{by~\eqref{dashcombs}--\eqref{defnSVN}}\\ &=
  Q_N\biggl(\sigma_m\otimes\sum_{n=1}^\infty \xi_n(x_n+y_n)\biggr) =
  Q_N(\sigma_m\otimes\xi)
\end{alignat*}
for each $m\in\N$. This implies that
\[ Q_N^{**}T_\xi^{**}(\Phi\otimes b_0) =
\underset{m\to\infty}{\text{w}^*\text{-}\lim}\,Q_N
T_\xi(\sigma_m\otimes b_0) =
\underset{m\to\infty}{\text{w}^*\text{-}\lim}\,
Q_N(\sigma_m\otimes\xi) = Q_N^{**}(\Phi\otimes\xi) \] because the
bidual of an operator~$T$ is a weak$^*$-continuous extension of~$T$.
Combining this identity with the diagrams~\eqref{commdiagramRead},
\eqref{defnTheta0}, and~\eqref{defnTcheck}, we can now
verify~\eqref{Uxi}:
\begin{align*} 
\Theta_0(\widecheck{T}_\xi)Ub_0 &=
\Theta_0(\widecheck{T}_\xi)\pi_{E_\text{R}}Q_N^{**}(\Phi\otimes b_0) =
\pi_{E_\text{R}}\widecheck{T}_\xi^{**}Q_N^{**}(\Phi\otimes b_0)\\ &=
\pi_{E_\text{R}}Q_N^{**}T_\xi^{**}(\Phi\otimes b_0) =
\pi_{E_\text{R}}Q_N^{**}(\Phi\otimes\xi) = U\xi,
\end{align*}
which completes the proof.
\end{proof}

\subsection*{Acknowledgements} We are grateful to 
Garth Dales (Lancaster) and Zinaida Lykova (Newcastle) for having
drawn our attention to the questions which motivated this paper, that
is, whether every extension of~$\mathscr{B}(E)$ which splits
algebraically also splits strongly, and whether~$\mathscr{B}(E)$ may
have homological bidimension at least two. We would also like to thank
Graham Jameson (Lancaster) and to acknowledge several helpful
conversations with Charles Read (Leeds).

\bibliographystyle{amsplain}

\end{document}